\documentclass[11pt]{amsart}%
\usepackage{graphicx}
\usepackage{amscd}
\usepackage{amsmath}
\usepackage{amsfonts}
\usepackage{amssymb}%
\providecommand{\U}[1]{\protect\rule{.1in}{.1in}}
\swapnumbers
\newtheorem{theorem}{Theorem}[section]
\theoremstyle{plain}

\newtheorem{definition}[theorem]{Definition}
\newtheorem{example}[theorem]{Example}

\numberwithin{equation}{section}

\makeatletter
\@namedef{subjclassname@2010}{\textup{2010} Mathematics Subject Classification}
\makeatother
\begin{document}
\title[Hausdorff Compactifications]{Hausdorff Compactifications}
\author{Matt Insall}
\address{Department of Mathematics and Statistics, Missouri University of Science and
Technology, 400 W. 12th St., Rolla, MO 65409-0020}
\email{insall@mst.edu}
\author{Peter A. Loeb}
\address{Department of Mathematics, University of Illinois,\\
1409 West Green St., Urbana, Ill. 61801, USA}
\email{PeterA3@AOL.com}
\author{Ma\l gorzata Aneta Marciniak}
\address{Department of Mathematics, Engineering and Computer Science, LaGuardia
Community College of the City University of New York, 31-10 Thomson Avenue,
Long Island City, NY}
\email{goga@mimuw.edu.pl}
\date{September 4, 2020}

\begin{abstract}
Previously, the authors used the insights of Robinson's nonstandard analysis
as a powerful tool to extend and simplify the construction of some
compactifications of regular spaces. They now show that any Hausdorff
compactification is obtainable with their method.

\end{abstract}
\maketitle

\section{Introduction}

In \cite{RegularCompact}, we extended to more general compactifications the
work on topological ends of Insall and Marciniak in \cite{InMal}. Fix an
appropriately saturated nonstandard extension of a regular, noncompact
topological space $\left(  X,\mathcal{T}\right)  $. In \cite{RegularCompact},
we showed that points needed to form a compactification can be formed from
equivalence classes of points not in the monad of any standard point. Any
equivalence relation on such points works, but not every compactification of
$X$ can be obtained this way. We show here that any Hausdorff compactification
of $X$ can be obtained using a natural equivalence relation. We conclude with
brief discussions of an application to a moduli space of triangles and to the
Martin boundary in potential theory and probability.

\section{General Compactifications}

First we review the construction in \cite{RegularCompact}. Also for
background, see \cite{Robinson}, \cite{RobinsonCompact}, and \cite{LW}. Let
$\left(  X,\mathcal{T}\right)  $ be a regular, noncompact topological space.

\begin{definition}
\label{DefinitionCompactification}By a compactification of $\left(
X,\mathcal{T}\right)  $, we mean a compact space $(Y,\mathcal{T}_{Y})$ such
that $X$ is a dense subset of $(Y,\mathcal{T}_{Y})$. We require here that the
identity mapping of $X$ into $Y$ is a homeomorphism from $\left(
X,\mathcal{T}\right)  $ to $X$ with the relative $\mathcal{T}_{Y}$-topology.
\end{definition}

\begin{example}%
\begin{rm}%
A simple example where the latter requirement fails is given by the set
$X=\left\{  1/n:n\in\mathbb{N}\right\}  \cup\left\{  0\right\}  $ with the
discrete topology. We set $Y=X$, but now the singleton $\{0\}$ is no longer an
open set. Its neighborhoods consist of all intervals in $X$ from $0$ to $1/n$,
$n\in\mathbb{N}$. Clearly, the set $X$ is a dense and continuous image in $Y$,
but $Y$ is not a compactification of $X$ in the sense of Definition
\ref{DefinitionCompactification}.%
\end{rm}%

\end{example}

We now fix a $\kappa$-saturated nonstandard extension of $\left(
X,\mathcal{T}\right)  $, where $\kappa$ is greater than the cardinality of the
topology $\mathcal{T}$ on $X$.

\begin{definition}
We call a point $x\in{}^{\ast}X$ \textbf{remote} if $x$ is not near-standard,
i.e., not in the monad of any standard point of $X$. Given an equivalence
relation on the set of remote points of {}$^{\ast}X$, we write $x\sim y$ if
$x$ and $y$ are remote and equivalent.
\end{definition}

Let $Y$ be the point set consisting of points of\ $\ X$, called
\textbf{s-points}, together with all equivalence classes of remote points,
where each such equivalence class is treated as a single point. We call each
such point of $Y$ an \textbf{r-point}. We supply ${}^{\ast}X$ with the
S-topology, that is the topology generated by the nonstandard extensions of
standard open subsets of $X$. Let $\varphi$ be the mapping from ${}^{\ast}X$
onto $Y$ that sends near-standard points to their standard parts and remote
points to their respective r-points. The \textbf{neighborhood filter base}
$B(y)$ at an r-point $y$ in $Y$ consists of all sets of the form
$\varphi\left(  {}^{\ast}U\right)  $, where $U$ is a standard open subset of
$X$ with nonstandard extension containing the entire equivalence class
corresponding to $y$ (whence $U\neq\varnothing$). The \textbf{neighborhood
filter base} $\mathcal{B}(x)$ at an s-point $x$ in $Y$ consists of all sets of
the form $\varphi\left(  ^{\ast}U\right)  $, where $U$ is a standard open
subset of $X$ with $x\in U$. For each point $p\in Y$, $\mathcal{B}(p)$ is in
fact a filter base (see \cite{RegularCompact}.) As usual, a set $O\subseteq Y$
is called \textquotedblleft open\textquotedblright\ if for each point $p\in
O$, there is an element $\varphi\left(  ^{\ast}U\right)  \in\mathcal{B}(p)$
with $\varphi\left(  {}^{\ast}U\right)  \subseteq O$. The collection of open
sets forms a topology on $Y$. Note that we have not made any claim about the
interior with respect to $Y$ of any member of any neighborhood filter base. We
let $\mathcal{T}_{Y}$ denote the topology on $Y$, i.e., the collection of open
sets, generated by the neighborhood filter bases. The following result is
established in \cite{RegularCompact} using the fact (see \cite{Todorov} and
\cite{Todorov2}) that ${}^{\ast}X$ with the S-topology is compact.

\begin{theorem}
The map $\varphi$ is a continuous surjection from ${}^{\ast}X$ onto $Y$,
whence, $Y$ is compact. Moreover, the point set $X$ is dense in $Y$ supplied
with the $\mathcal{T}_{Y}$-topology, and $\mathcal{T}$ is stronger than, or
equal to, the relative $\mathcal{T}_{Y}$-topology on $X$.
\end{theorem}

Recall that all sets in $\mathcal{T}$ are themselves in $\mathcal{T}_{Y}$ if
$\mathcal{T}$ is a locally compact topology on $X$. The following is an
example where the $\mathcal{T}$-topology on $X$ is strictly stronger than the
relative $\mathcal{T}_{Y}$-topology.

\begin{example}%
\begin{rm}%
Let $X$ be the rational numbers in the interval $[0,1]$. A point $x$ in
$^{\ast}X$ is remote if and only if $x$ is in the monad of an irrational point
in $[0,1]$. Put all remote points in the same equivalence class, so there is
only one $r$-point, denoted by $\alpha$, in $Y$. If $U$ is a nonempty open set
in $X$, then {}$^{\ast}U$ contains remote points. In order for there to be a
set $V$ in $\mathcal{T}_{Y}$ for which the restriction to $X$ is $U$, it is
necessary that $V$ contains an element of the neighborhood filter base of
$\alpha$. That is, $V$ must contain a set $\varphi\left[  ^{\ast}W\right]  $
where $W$ is in $\mathcal{T}$, and $^{\ast}W$ contains every remote point in
$^{\ast}X$. Such a subset $W$ of $X$ must contain every point of $X$ except
perhaps those in a compact subset of $X$. Therefore, the point set $Y$
consists of the points of $X$ together with the point $\alpha$, and the
nonempty elements of $\mathcal{T}_{Y}$ are the complements in $Y$ of compact
subsets of $X$.%
\end{rm}%

\end{example}

In the next section, we show that any Hausdorff compactification of a
Hausdorff space $\left(  X,\mathcal{T}\right)  $ can be produced with an
equivalence relation on the remote points of {}$^{\ast}X$. First, however, we
give an example showing that this may not be true if the compactification is
not Hausdorff. It is also an example of a Hausdorff space $\left(
X,\mathcal{T}\right)  $ that is not locally compact, but still forms an open
set in a compactification.

\begin{example}%
\begin{rm}%
We modify Example 2 on Page 630 from \cite{Compactifications}. Let $X$ be the
subset of the plane given by%
\[
X=\left\{  (x,0):x\text{ rational, }0\leq x\leq1\right\}  .
\]
To form the compactification $Y$, we adjoint to $X$ the subset $\Delta$ in the
plane given by%
\[
\Delta=\left\{  (x,1):0\leq x\leq1\right\}  .
\]
A typical neighborhood of a point $\left(  x_{0},0\right)  $ in $X$ is given
by the relative plane topology. That is, it is given by a constant
$\varepsilon>0$ and has the form%
\[
\{\left(  x,0\right)  \in X:|x-x_{0}|<\varepsilon\}.
\]
A typical neighborhood of a point $\left(  x_{1},1\right)  \in\Delta$ is given
by a constant $\delta>0$ and has the form%
\[
\{\left(  x,1\right)  \in\Delta:|x-x_{1}|<\delta\}\cup\left\{  (x,0)\in
X:|x-x_{1}|<\delta\right\}  .
\]
Clearly, $X$ is dense in $Y=X\cup\Delta$, and $Y$ is not Hausdorff. Moreover,
$Y$ is compact, since any net has a cluster point in $\Delta$. The
compactification $Y$ cannot be obtained using an equivalence relation on the
remote points of ${}^{\ast}X$ when the neighborhoods are formed as in
\cite{RegularCompact}. To see this, suppose $\left(  r,0\right)  $ is a remote
point in $^{\ast}X$ that is in the equivalence class forming the point
$\left(  0,1\right)  $ in $\Delta$. Then $r$ is in the monad of a standard
irrational $s$ in $[0,1]$. If the construction from \cite{RegularCompact}
works here, then there must be a standard open set $U$ in $X$ such that
$\left(  r,0\right)  $ is in $^{\ast}U$, and every near-standard point in
{}$^{\ast}U$ \ has $x$-coordinate less than $s/2$. But this is impossible.%
\end{rm}%

\end{example}

\section{Hausdorff Compactifications}

Assume that $\left(  X,\mathcal{T}\right)  $ is a Hausdorff \ space, and let
$\left(  Z,\mathcal{T}_{Z}\right)  $ be a Hausdorff compactification of
$\left(  X,\mathcal{T}\right)  $. That is, $X$ is dense subset of the compact
Hausdorff space $Z$, and the mapping from $X$ to $X$ as a subset of $Z$ is a homeomorphism.

\begin{definition}
Remote points $p$ and $q$ in $^{\ast}X$ are equivalent, i.e., $p\sim q$, if
$p$ and $q$ are in the monad of a point $z\in Z\setminus X$. Let $\left(
Y,\mathcal{T}_{Y}\right)  $ be the compact space produced with this
equivalence relation. Denote by $F$ the mapping from $Z$ to $Y$ that is the
identity mapping from $X$ as a subset of $Z$ to $X$ as a subset of $Y$, and
maps each point $z\in Z\setminus X$ to the r-point in $Y$ formed from the
equivalence class formed by those points in $^{\ast}X$ that are in the monad
of $z$.
\end{definition}

\begin{theorem}
\label{main}The map $F$ is a bijection, and indeed, a homeomorphism from
$\left(  Z,\mathcal{T}_{Z}\right)  $ onto $\left(  Y,\mathcal{T}_{Y}\right)
$. It follows that any Hausdorff compactification of $X$ can be obtained from
an equivalence relation on the remote points of $^{\ast}X$.
\end{theorem}

%

\begin{proof}%
It is clear that $F$ is bijective. As is well known, it is sufficient to show
that $F^{-1}$ is continuous. For then, the inverse image using $F$ of an open
set in $Y$ will be the\ open complement in $Z$ of the compact image of a
closed, therefore compact set in $Y$. Fix $z\in Z$, and an open neighborhood
$U$ of $z$. By the regularity of $\left(  Z,\mathcal{T}_{Z}\right)  $, there
is an open neighborhood $V$ of $z$ for which the closure is contained in $U$.
Let $O=V\cap X$. Since $\left(  Z,\mathcal{T}_{Z}\right)  $ is a
compactification of $\left(  X,\mathcal{T}\right)  $, $O\in\mathcal{T}$.
Moreover, $^{\ast}O$ contains $F(z)$, and $\varphi\left(  ^{\ast}O\right)  $
is a closed neighborhood of $F(z)$ that maps, using $F^{-1}$, into $U$. It
follows that $F^{-1}$ is continuous.%
\end{proof}%

\begin{example}%
\begin{rm}%
As an application to moduli spaces, we compactify the space of similar
triangles in $\mathbb{R}^{2}$ discussed by Madeline Brandt in \cite{Brandt}.
Here, two triangles are equivalent with respect to this space of triangles if
they are similar. Each\ such equivalence class is represented by the member
with a leg of shortest length on the interval in the $y$-axis from vertex
$(0,0)$ to vertex $(0,1)$, and with the third vertex in the parameter space%
\[
S=\left(  (x,y)\in\mathbb{R}^{2}:x>0,\ y\geq1/2,\ x^{2}+\left(  y-1\right)
^{2}\geq1\right)  .
\]
Thinking of $S$ as a subset of the complex plane, the remote points are the
points in the nonstandard extension of $S$ with argument in the monad of
$\pi/2$ together with points with unlimited modulus and complex argument in
the interval {}$^{\ast}(0,\pi/2)$. To form a compactification, we can call
equivalent all remote points with argument in the monad of $\pi/2$, and
represent that equivalence class with one point at $2i$. On the other hand, we
can break up that equivalence class so there is just one point for members
with unlimited modulus, and represent that set with the nonnegative $y$-axis.
We can then call two remote points with infinitesimally close, limited moduli
and argument in the monad of $\pi/2$ equivalent. That equivalence class is
then represented by the line from the origin to the unique standard part of
the members. Each remaining remote point is the third vertex of a nonstandard
triangle. That triangle produces a \textquotedblleft standard
part\textquotedblright\ in $S$ consisting of the standard parts of those
points of limited modulus in the triangle. We call two such remote points
equivalent if their arguments have the same standard part $\theta<\pi/2$. The
corresponding triangles have the same standard part in $S$. The corresponding
equivalence class can be represented by that standard part, which is the
figure formed by the line from $(0,0)$ to $(0,1)$ together with the infinite
line segment with endpoint $(0,0)$ forming an angle $\theta$ with the $x$-axis
and the parallel infinite line segment with endpoint $(0,1)$.%
\end{rm}%

\end{example}

\begin{example}%
\begin{rm}%
The Martin compactification (see \cite{Doob}) is an important construction in
potential theory; it can be seen as an application of Theorem 3.14 in
\cite{RegularCompact}. Martin's boundary is also of great importance in
probability theory. Motivated by our Theorem \ref{main}, the second author
obtained a probability formulated equivalence relation in \cite{MartinLux}
that yields the Martin compactification for a number of illuminating examples.
His approach, \textquotedblleft looking inside\textquotedblright\ a domain,
only makes sense when one can speak of points that are neither points of an
existing boundary nor points in a compact subset of the domain.%
\end{rm}%

\end{example}

\end{document}